\documentclass[12pt]{article}
 \usepackage[margin=1in]{geometry} 
\usepackage{amsmath,amsthm,amssymb,amsfonts}
 \usepackage{amssymb}

\usepackage[utf8]{inputenc}

\usepackage{graphicx}    
\usepackage{rotating}
\usepackage{color}
\usepackage{epsfig}
\usepackage{amsxtra}     
\usepackage{amsthm}
\usepackage{amssymb}
\usepackage{amsfonts}
\usepackage{pgfplots}
\pgfplotsset{compat=1.12}

\usepgfplotslibrary{fillbetween}
\usetikzlibrary{patterns}
\usepackage{hyperref}
 \usepackage{csquotes}
\usepackage[utf8]{inputenc}
\usepackage[english]{babel}
\usepackage{xcolor}
\usepackage{biblatex}
\numberwithin{equation}{section}
\newtheorem{theorem}{Theorem}[section]
\newtheorem{lemma}[theorem]{Lemma}

\newtheorem{Prop}{Proposition}[section]

\providecommand{\keywords}[1]
{
  \small	
  \textbf{\textit{Keywords---}} #1
}
\title{Application of Orlicz regularity theory in two dimensional chemotaxis systems with logistic sources }
\author{ Minh Le \\
 Michigan State University\\
  \texttt{leminh2@msu.edu} }
\date{\today}

\begin{document}
\maketitle
\begin{abstract}
 The objective is to investigate the global existence of solutions for a degenerate chemotaxis system with logistic sources in a two-dimensional domain. It is demonstrated that the inclusion of logistic sources can exclude the occurrence of blow-up solutions, even in the presence of superlinear growth in the cross-diffusion rate. Our proof relies on the application of elliptic and parabolic regularity in Orlicz spaces and variational approach.
\end{abstract}
\section{Introduction}
We consider the following system arising from chemotaxis in a smooth bounded domain $\Omega \subset \mathbb{R}^2$: 
\begin{equation} \label{1}
    \begin{cases} 
        u_t &= \nabla \cdot (D(v)\nabla(u)) -  \nabla \cdot (S(v)u \ln^\alpha (u+e) \nabla v) +ru -\mu u^{2} \\
        \eta v_t&= \Delta v -v +u,
    \end{cases}
\end{equation}
where $\eta \in \left \{ 0,1 \right \}$, $r\in \mathbb{R}$, $\mu \geq 0$, $\alpha \geq 0$, and
\begin{align} \label{diffusion-condition}
    0<D \in C^2([0,\infty)) \quad \text{ and } S \in C^2([0,\infty))\cap W^{1,\infty} ((0,\infty)) \text{ such that } S' \geq 0.
\end{align}
The system \eqref{1} is complemented with nonnegative, initial conditions in $W^{1,\infty}(\Omega)$ not identically zero:
\begin{align} \label{initial-data}
    u(x,0)=u_0(x), \qquad v(x,0)=v_0(x), \qquad \text{with } x\in \Omega,
 \end{align}
and homogeneous Neumann boundary condition are imposed as follows:
\begin{equation} \label{boundary-data}
    \frac{\partial u}{\partial \nu } = \frac{\partial v}{\partial \nu }  = 0, \qquad x\in \partial \Omega,\, t \in (0, T_{\rm max}),
\end{equation}
where $\nu$ denotes the outward normal vector. \\
The simplest version of the system \eqref{1} occurs when $D\equiv S \equiv 1$, $\alpha=0$, and $\eta= r=\mu =0$, known as the Keller-Segel system \cite{Keller} has been extensively studied since the 1970s due to its intriguing phenomenon, known as critical mass. This phenomenon states that if the initial mass is less than a certain value, solutions are global and bounded. However, if it is greater than that value, solutions blow up in finite or infinite time.

The presence of logistic sources in the first equation of the Keller-Segel system was proven in \cite{Tello+Winkler} to be sufficient to exclude blow-up solutions when $\eta =0$, and $\mu > \frac{n-2}{n}$ for any dimension $n \geq 2$. This result was later improved for the critical parameter case when $\mu =\frac{n-2}{n}$ by \cite{Hu+Tao,KA, Tian5}. In \cite{Winkler-logistic}, it was proven that the quadratic degradation term, $-\mu u^2$ with sufficiently large $\mu$ can ensure the existence of global boundedness solutions for the fully parabolic system when $\eta>0$.
Numerous research papers concerning blow-up prevention by logistic sources in various chemotaxis systems can be found in \cite{Jian+Xiang, Lankeit2, LMLW, Li+Wang, LMZ, Tao+Winkler, MW2011, YCJZ, EOJ, MJ, Minh}.

There is also a rich literature for the system \eqref{1} in two spatial dimensional domains. Many studies revolve around the question: "Is the quadratic logistic degradation term optimal in preventing blow-up?" One of the first answers to this question, given in \cite{Tian4, Tian2}, is "no" when $D \equiv S \equiv 1$ and $\alpha =0$. In fact, it was demonstrated that the sub-quadratic degradation term, $\frac{-\mu u^2}{\ln^p(u+e)}$, for $p\in (0,1]$ is adequate to ensure global boundedness of solutions. Additionally, the quadratic degradation is a very effective term to deal with general conditions for $D$ and $S$. In \cite{Jin+Wang}, when $\alpha =\eta =0$, and the boundedness condition for $\frac{S^2}{D}$ can guarantee the global boundedness of solutions. Similar results were proved in \cite{MW2022} for $\alpha=0$ and $\eta =1$ by considering a general diffusion term $D>0$, however, restricting $S \in C^2([0, \infty))\cap W^{1, \infty}((0, \infty))$. Additionally, by employing the Orlicz regularity for parabolic equations established in \cite{MW2022}, it was shown in \cite{Minh2} that the sub-quadratic degradation term can ensure the exclusion of blow-up solutions when $\alpha =0$ and $\eta=1$.

Our main objective is to show that the quadratic logistic degradation term can effectively prevent blow-up for both elliptic-parabolic and fully parabolic degenerate chemotaxis models with superlinear growth in the cross-diffusion rate where $\alpha>0$. To be more precise, our main result reads as follows:

\begin{theorem} \label{degenerate}
    Suppose that $\eta =0$ and $\alpha \in (0,1)$ then the system \eqref{1} under the assumptions  \eqref{diffusion-condition}, \eqref{initial-data} and \eqref{boundary-data} admits a global classical solution $(u,v)$ in $  \left [C^0 \left ( \bar{\Omega}\times [0,\infty) \right ) \cap C^{2,1} \left ( \bar{\Omega}\times (0,\infty)  \right ) \right ]^2$ 
   such that $u>0$ and $v>0$ in $\bar{\Omega}\times   (0, \infty)$. Furthermore, this solution is bounded in the sense that 
    \begin{align}
        \sup_{t>0} \left \{ \left \| u(\cdot,t) \right \|_{L^\infty(\Omega)}+\left \| v(\cdot,t) \right \|_{W^{1,\infty}(\Omega)} \right \} <\infty.
    \end{align}

\end{theorem}
For fully parabolic cases, we have the following theorem:

\begin{theorem} \label{degenerate-p}
    Suppose $\eta =1$ and $\alpha \in (0,\frac{1}{2})$ then the system \eqref{1} under the assumptions \eqref{diffusion-condition}, \eqref{initial-data} and \eqref{boundary-data} admits a global classical solution $(u,v)$ with
      \begin{equation*}
        \begin{cases}
            u \in C^0 \left ( \bar{\Omega}\times [0,\infty) \right ) \cap C^{2,1} \left ( \bar{\Omega}\times (0,\infty)  \right ) \qquad \text{and} \\
            v \in \cap_{q>2} C^0\left ( [0,\infty); W^{1,q}(\Omega) \right )\cap  C^{2,1} \left ( \bar{\Omega}\times (0,\infty)  \right ).
        \end{cases}
    \end{equation*}
    Furthermore, this solution is bounded in the sense that 
    \begin{align}
        \sup_{t>0} \left \{ \left \| u(\cdot,t) \right \|_{L^\infty(\Omega)}+\left \| v(\cdot,t) \right \|_{W^{1,\infty}(\Omega)} \right \} <\infty.
    \end{align}
\end{theorem}
The proof of the main results can be summarized into three steps:
\begin{enumerate}
    \item \textbf{Derive an initial estimate for solutions: }  
    \begin{align*}
        \sup_{t \in (0,T_{\rm max})}\int_\Omega  u \ln^k(u+e) +\eta |\nabla v|^2 <\infty, \qquad \text{for some } k \geq 1.
    \end{align*}
    To accomplish this, we adapt and modify the argument presented in \cite{Minh2}[Lemma 4.1] for the proof of Lemma \ref{l1} and \ref{l1-p}.
    \item \textbf{Address the degeneracy of the diffusion term:} Eliminate the degeneracy of the diffusion term by employing elliptic and parabolic regularity in Orlicz spaces. The proof of the elliptic part is provided in Lemma \ref{L.2}, and we apply the parabolic part as established in \cite{MW2022}.
    \item \textbf{Establish $L^p$ bounds for the solution:}
    Lemma \ref{l2} and \ref{l2-p} establish $L^p$ bounds for the solution for any $p>1$. The primary challenge lies in incorporating the term $\int_\Omega u^p \ln^\alpha(u+e)$ into the diffusion term. Overcoming this difficulty involves the utilization of logarithmically refined Gagliardo-Nirenberg interpolation inequalities, as established in \cite{Winkler_preprint}.
\end{enumerate}

This paper is structured as follows. In Section \ref{preliminairies}, we revisit local existence results for both elliptic-parabolic and fully parabolic models, along with key inequalities used in subsequent sections. We also provide results on regularity in Orlicz spaces. Section \ref{E} presents a priori estimates, including $L \ln^k L$ and $L^p$ estimates for solutions of elliptic-parabolic models when $\eta=0$, and includes the proof of Theorem \ref{degenerate}. Section \ref{P} follows a similar framework but addresses the fully parabolic case when $\eta=1$ and includes the proof of Theorem \ref{degenerate-p}.

\section{Preliminaries} \label{preliminairies}
By employing fixed point arguments and applying standard theories of elliptic and parabolic regularity, we can establish the local existence and uniqueness of non-negative classical solutions to the system \eqref{1}. Our initial step involves establishing the local existence of solutions for parabolic-elliptic chemotaxis models, and we achieve this by adapting the method presented in \cite{TW}[Theorem 2.1].
\begin{lemma} \label{local-existence-e}
    Let $\Omega \subset \mathbb{R}^2$ be a bounded domain with smooth boundary and that \eqref{diffusion-condition}, \eqref{initial-data}, and \eqref{boundary-data} hold. Then there exist $T_{\rm max}\in (0,\infty]$ and functions $(u,v)$ in   $  \left [C^0 \left ( \bar{\Omega}\times [0,T_{\rm max}) \right ) \cap C^{2,1} \left ( \bar{\Omega}\times (0,T_{\rm max})  \right ) \right ]^2$ 
    such that $u>0$ and $v>0$ in $\Bar{\Omega}\times (0,\infty)$, that $(u,v)$ solves \eqref{1} classically in $\Omega \times (0,T_{\rm max})$, and that  
    \begin{equation}
        \text{if }T_{\rm max}<\infty, \quad \text{then } \limsup_{t\to T_{\rm max}} \left \{ \left \| u(\cdot,t) \right \|_{L^\infty(\Omega)} +\left \| v(\cdot,t) \right \|_{W^{1,\infty}(\Omega)} \right \} = \infty.
    \end{equation}
\end{lemma}
The local existence of solutions for fully parabolic models can be attained by modifying and adjusting the proof in \cite{Winkler-logistic}[Lemma 1.1] or referring to \cite{Winkler-Horstmann, Lankeit-2017}
\begin{lemma} \label{local-existence}
    Let $\Omega \subset \mathbb{R}^2$ be a bounded domain with smooth boundary and that \eqref{diffusion-condition}, \eqref{initial-data}, and \eqref{boundary-data} hold. Then there exist $T_{\rm max}\in (0,\infty]$ and functions 
    \begin{equation}
        \begin{cases}
            u \in C^0 \left ( [0,T_{\rm max});C^0(\bar{\Omega}
            ) \right ) \cap C^{2,1} \left ( \Bar{\Omega}\times (0,T_{\rm max}) \right )\text{ and} \\
            v \in \bigcap_{q>2} C^0 \left ( [0,T_{\rm max}); W^{1,q}(\Omega
            ) \right )\cap C^{2,1} \left ( \Bar{\Omega}\times (0,T_{\rm max}) \right )
        \end{cases}
    \end{equation}
    such that $u>0$ and $v>0$ in $\Bar{\Omega}\times (0,\infty)$, that $(u,v)$ solves \eqref{1} classically in $\Omega \times (0,T_{\rm max})$, and that  
    \begin{equation}
        \text{if }T_{\rm max}<\infty, \quad \text{then } \limsup_{t\to T_{\rm max}} \left \{ \left \| u (\cdot,t)\right \|_{L^\infty(\Omega)} +\left \| u(\cdot,t) \right \|_{W^{1,\infty}(\Omega)} \right \} = \infty.
    \end{equation}
\end{lemma}

The following lemma is a direct consequence of \cite{Winkler_preprint}[Corollary 1.2]
\begin{lemma} \label{LGN} If $\Omega \subset \mathbb{R}^2$ is a bounded domain with smooth boundary, then for each $p>0$ and $\gamma\geq 0$ there exists $C=C(p,\gamma)>0$ with the property that whenever $\phi \in C^1 (\bar{\Omega})$ is positive in $\bar{\Omega}$
    \begin{align}\label{LGN.1}
        \int_\Omega \phi^ {p+1} \ln^{\gamma}(\phi+e) \leq C \left ( \int_\Omega \phi \ln^{\gamma}(\phi+e) \right ) \left ( \int_\Omega |\nabla \phi^{\frac{p}{2}}|^2 \right  ) +C \left ( \int_\Omega  \phi \right )^p \left ( \int_\Omega \phi \ln^{\gamma}(\phi+e) \right ).
    \end{align}
\end{lemma}
As a consequence, the next lemma serves as a key to digest the superlinear growth in the cross-diffusion rate into the diffusion term.
\begin{lemma}\label{ILGN} Assume that $\Omega \subset \mathbb{R}^2$ is a bounded domain with smooth boundary  and $p>0$, $\gamma>\xi \geq 0$. For each $\epsilon>0$, there exists $C=C(\epsilon,\xi,\gamma)>0$ such that 
    \begin{align}\label{ILGN.1}
        \int_\Omega \phi^ {p+1} \ln^{\xi}(\phi+e) \leq \epsilon \left ( \int_\Omega \phi \ln^{\gamma}(\phi+e) \right ) \left ( \int_\Omega |\nabla \phi^{\frac{p}{2}}|^2 \right  ) +\epsilon \left ( \int_\Omega  \phi \right )^p \left ( \int_\Omega \phi \ln^{\gamma}(\phi+e) \right ) +C.
    \end{align}
\end{lemma}
\begin{proof}
    Since $\gamma >\xi \geq 0$, one can verify that for any $\delta>0$, there exists $c_1=c(\delta, \xi, \gamma)>0$ such that for any $a \geq 0$ we have
    \begin{align}
        a^{p+1} \ln^\xi (a+e) \leq \delta a^{p+1} \ln^ \gamma (a+e)+c_1.
    \end{align}
    This entails that
    \begin{align}
        \int_\Omega \phi^ {p+1} \ln^{\xi}(\phi+e) \leq \delta \int_\Omega \phi^ {p+1} \ln^{\gamma}(\phi+e)+c_1|\Omega|.
    \end{align}
    Now for any fixed $\epsilon$, we choose $\delta = \frac{\epsilon}{C}$ where $C$ as in Lemma \ref{LGN}, and apply \eqref{LGN.1} to have the desire inequality \eqref{ILGN.1}.
\end{proof}

The following Lemma  \cite{Tao+Winkler2}[Lemma A.1] provides a useful pointwise estimate for Green's function of $-\Delta +1$.
\begin{lemma} \label{L.1 }
Suppose that $\Omega \subset \mathbb{R}^2$ is a bounded domain with smooth boundary, and let $G$ denote Green's function of $-\Delta +1$ in $\Omega$ subject to Neumann boundary conditions. Then there exist $A> diam(\Omega)$ and $K>0$ such that 
\begin{align}
    |G(x,y)| \leq K \ln \frac{A}{|x-y|} \qquad \text{ for all } x,y \in \Omega \text{ with }x\neq y.
\end{align}
\end{lemma}
By the pointwise estimate for the Green's function and Legendre transform, we can derive a $L^\infty$ bound for solutions of \eqref{1} when $\eta=0$, and therefore eliminate the degeneracy of diffusion term. 
\begin{lemma} \label{L.2}
    Let $\Omega \subset \mathbb{R}^2$ be a bounded domain with smooth boundary. Suppose that the nonnegative function $f$ in $L^2(\Omega)$ satisfies  
    \begin{align}
        \int_\Omega f \ln (f+e) \leq M
    \end{align}
    and $w$ is a solutions of 
    \begin{equation}
        \begin{cases}
            -\Delta w +w = f, \quad x \in \Omega\\
            \frac{\partial u}{\partial \nu } = 0, \quad x\in \partial \Omega,
        \end{cases}
    \end{equation}
    then we have 
    \begin{align*}
        \left \| w \right \|_{L^\infty(\Omega)} \leq C,
    \end{align*}
    where $C=C(M)>0$
\end{lemma}
\begin{proof}
By using the Green's function $G$ of $-\Delta +1$ in $\Omega$, Lemma \ref{L.1 } and the inequality that
\begin{align*}
  ab \leq a\ln a +e^{b-1}, \qquad \text{for all }  a, b \geq 0,
\end{align*}
we deduce that
\begin{align}
    w(x) &= \int_\Omega G(x,y)f(y)\, dy \notag \\
    &\leq K \int_\Omega \ln \frac{A}{|x-y|}f(y )\, dy  \notag \\
    &\leq K\int_\Omega  f(y) \ln f(y) + K\int_\Omega e^{\ln\frac{A}{|x-y|}-1} \notag \\
    &\leq KM +\frac{AK}{e} \int_\Omega \frac{1}{|x-y|}\,dy \notag \\
    &\leq KM +\frac{AK}{e}{diam(\Omega)}.
\end{align}    
\end{proof}

The following lemma states the results in Sobolev spaces, which is an immediate application of \cite{Winkler-2010}[Lemma 1.3]. 
\begin{lemma} \label{Para-Reg}
Assuming that $\Omega \subset \mathbb{R}^n$, with $n\geq 2$ is an open bounded domain, $V_0 \in  W^{1,q}(\Omega)$ with $q>n$, $f \in C \left ( \bar{\Omega} \times [0,T) \right )   $, and $V\in C\left ( \bar{\Omega} \times[0,T) \right)\cap C^{2,1}\left ( \bar{\Omega} \times(0,T) \right) \cap C \left ( [0,T); W^{1,q}(\Omega) \right)$ is a classical solution to the following system
\begin{equation}\label{parabolic-equation}
    \begin{cases}
     V_t = \Delta V  - a V + f &\text{in } \Omega \times (0,T), \\ 
\frac{\partial V}{\partial \nu} =  0 & \text{on }\partial \Omega \times (0,T),\\ 
 V(\cdot,0)=V_0   & \text{in } \Omega
    \end{cases} 
\end{equation}
where $a>0$ and $T\in (0,\infty]$. If $f \in L^\infty \left ( (0,T);L^p(\Omega) \right ) $, for some $p>n$ then $V  \in L^\infty \left ( (0,T);W^{1,\infty}(\Omega) \right )$.
\end{lemma}
The subsequent parabolic regularity result is crucial in scenarios of strong degeneracy, specifically when $\inf_{s\geq 0}D(s) =0$. It has been established that the equation \eqref{parabolic-equation} exhibits a globally bounded solution, contingent upon an appropriate slow growth condition on $f$. To be precise, the following proposition directly applies Corollary 1.3 from \cite{MW2022} with $n=2$.
\begin{Prop} \label{Orlicz-regularity}
For each $a>0$, $q>2$, $K>0$ and $\tau>0$, there exist $C(a,q,K,\tau)>0$ such that if $T\geq 2\tau $, $f \in C^0(\Bar{\Omega}\times[0,T])$, and  $V \in C^0 (\Bar{\Omega}\times [0,T]) \cap C^{2,1}(\Bar{\Omega}\times (0,T) ) \cap C^0 ([0,T);W^{1,q}(\Omega))$ are such that \eqref{parabolic-equation} is satisfied with
\begin{align}
    \int_t ^{t+\tau} \int_\Omega |f|^2 \ln^\gamma(|f|+e) <K\text{ for all } t \in ( 0, T-\tau ), 
\end{align}
where $\gamma > 1$, and  \[\left \|V_0 \right \|_{ W^{1,q}(\Omega)}<K, \]
then
\begin{align}
    |V(x,t)| \leq C(a,q,K,\tau) \quad \text{ for all }(x,t)\in \Omega \times (0,T).
\end{align}
\end{Prop}

\section{Elliptic-Parabolic system} \label{E}
Let us begin this section with an $L\ln^k L$ estimate for solutions of \eqref{1}. The key approach in the proof is grounded in the Lyapunov functional method. While a standard estimate in two-dimensional domains is often considered when $k=1$, we aim to enhance it by exploring the case where $k \geq 1$. The inspiration is drawn from the construction of a Lyapunov functional in an unconventional manner, as introduced in \cite{Tian2}. This idea has been adapted and refined in \cite{Minh2} for addressing two-dimensional chemotaxis models with a degenerate diffusion term, and in \cite{Minh3} for two-species with two chemicals, although the logistic source appears only in one of the two density population equations.
\begin{lemma} \label{l1}
    Under the assumptions in Theorem \ref{degenerate}, for any $k\geq 1$, we have that 
    \begin{align} \label{l1.1'}
        \sup_{t\in (0, T_{\rm max })} \int_\Omega u(\cdot,t) \ln^k(u(\cdot,t)+e) <\infty.
    \end{align}
\end{lemma}
\begin{proof}
We define 
 \[
 I(t) := \int_\Omega u \ln^k(u+e)
 \]
 and differentiate $I(\cdot)$ to obtain
 \begin{align}\label{l1-p.2}
     I'(t) &= \int_\Omega \left \{ \ln^k{(u+e)} +ku \frac{\ln^{k-1}{(u+e)}}{u+e} \right \} \left ( \nabla \cdot \left ( D(v)\nabla u -u S(v) \ln^\alpha (u+e) \nabla v \right ) + f(u) \right ) \notag \\
     &=-k\int_\Omega \frac{D(v)\ln^{k-1}(u+e)}{u+e}|\nabla u|^2 -k(k-1) \int_\Omega \frac{D(v)u \ln ^{k-2}(u+e)}{(u+e)^2} |\nabla u|^2 \notag \\
     &-k \int_\Omega \frac{eD(v)\ln^{k-1}(u+e)}{(u+e)^2 }|\nabla u|^2 +\int_\Omega S(v) \nabla \phi(u)\cdot \nabla v \notag \\
     &+ \int _\Omega  \left \{ \ln^k{(u+e)} +ku \frac{\ln^{k-1}{(u+e)}}{u+e} \right \} (ru -\mu u^2) \notag \\
     & \leq \int_\Omega S(v) \nabla \phi(u)\cdot \nabla v  + \int _\Omega  \left \{ \ln^k{(u+e)} +ku \frac{\ln^{k-1}{(u+e)}}{u+e} \right \} (ru -\mu u^2 )
 \end{align}
 where 
 \begin{align}\label{l1-p.3}
     \phi(l)&: = \int_0^l \left \{ \frac{ks \ln^{k+\alpha-1}(s+e)}{s+e} + \frac{k(k-1) u^2 \ln^{k+\alpha -2}(s+e)}{(s+e)^2} +\frac{kes\ln^{k+\alpha-1}}{(s+e)^2} \right \}\,ds \notag \\
     &\leq c_1 l \ln^{k+\alpha-1}(l+e), \qquad \text{for all } l\geq 0,  
 \end{align}
with $c_1= k^2+k$. By using integration by parts, taking into account the condition $S' \geq 0$ and applying elementary inequalities, we obtain that 
\begin{align}\label{l1.1}
   \int_\Omega S(v) \nabla \phi(u)\cdot \nabla v&=-\int_\Omega S(v) \phi(u) \Delta v - \int_\Omega S'(v) \phi(u)|\nabla v|^2 \notag \\
   &\leq \left \| S \right \|_{L^\infty((0, \infty))} \int_\Omega \phi(u)u \notag \\
   &\leq c_1\left \| S \right \|_{L^\infty((0, \infty))} \int_\Omega u^2 \ln^{k+\alpha-1} (u+e)\notag \\
   &\leq \frac{\mu}{4} \int_\Omega u^2 \ln^{k}(u+e) +c_2,
\end{align}
where $c_2=C(\mu, \alpha,k)>0$ and the last inequality comes from the fact that $k +\alpha-1 <k$ when $\alpha<1$ and the inequality that for any $\delta>0$, there exist $A=c(\delta)>0$ such that 
\begin{align} \label{l1-p.5}
    s^{a_1}\ln^{b_1}(s+e) \leq \delta s^{a_2}\ln^{b_2}(s+e) +A, \qquad \text{for all } s\geq 0,
\end{align}
where $a_1, a_2, b_1,b_2$ are positive numbers such that $a_1 <a_2$. To handle the last term of \eqref{l1-p.2} , we make use of again \eqref{l1-p.5} to obtain
\begin{align}\label{l1-p.6}
  \int_\Omega  \left ( \ln^k{(u+e)} +k \frac{\ln^{k-1}{(u+e)}}{u+e} \right )(ru-\mu u^2) &\leq r \int_\Omega u\ln^k(u+e)  \notag \\&+r \int_\Omega k\ln^{k-1}(u+e)-\mu \int_\Omega u^2\ln^{k}(u+e) \notag \\
   & \leq  -\frac{\mu}{4} \int_\Omega u^2 \ln^{k}(u+e) +c_3,
\end{align}
where $c_3=C(\mu)>0$ .The inequality \eqref{l1-p.5} also implies that there exists $c_{4}=C(\mu)>0$ such that
\begin{align} \label{l1-p.8'}
    \int_\Omega u \ln^{k}(u+e) \leq \frac{\mu}{4}  \int_\Omega u^2 \ln^{k}(u+e)+c_{4}.
\end{align}
Collecting \eqref{l1-p.2}, \eqref{l1.1}, \eqref{l1-p.6}, and \eqref{l1-p.8'} yields
\begin{align}
    I'(t) +I(t) \leq c_5,
\end{align}
where $c_5=c_2+c_3+c_4$. Finally, we apply Gronwall's inequality to prove \eqref{l1.1'}.
\end{proof}

We will establish an $L^p$ estimate for the solution in the following lemma. When employing the standard testing approach commonly used in chemotaxis, controlling the term $\int_\Omega u^{p+1}\ln^\alpha (u+e)$ proves challenging using the diffusion term $-\int_\Omega |\nabla u^{\frac{p}{2}}|^2$ and the global boundedness of $\int_\Omega u$. To overcome this difficulty, the key idea is to utilize the bound $\int_\Omega u \ln^k(u+e)$ instead of $\int_\Omega u$ and the logarithmically refined Gagliardo-Nirenberg interpolation inequality in Lemma \ref{ILGN}.

\begin{lemma} \label{l2}
    Under the assumptions in Theorem \ref{degenerate}, for any $p>1$, we have that 
    \begin{align}
        \sup_{t\in (0, T_{\rm max })} \int_\Omega u^p(x,t)  dx<\infty.
    \end{align}
\end{lemma}
\begin{proof}
    
    By integration by parts, we have
  \begin{align} \label{l2.1}
        \frac{1}{p}\frac{d}{dt} \int_\Omega u^p &= \int_\Omega u^{p-1} \left ( \nabla (D(v) \nabla u)- \nabla \cdot ( S(v) u \ln^\alpha(u+e) \nabla v) +ru -\mu u^2 \right ) \notag \\
        &= - \frac{2(p-1)}{p} \int_\Omega D(v)|\nabla u^{\frac{p}{2}}|^2 +(p-1) \int_\Omega S(v) u^{p-1} \ln^\alpha (u+e) \nabla u \cdot \nabla v \notag \\
        &+r\int_\Omega u^p -\mu \int_\Omega u^{p+1} 
    \end{align}
       From Lemma \ref{l1}, there exist a constant $M>0$ such that
    \begin{align}
        \int_\Omega u(\cdot,t) \ln (u(\cdot,t)+e) \leq M, \qquad \text{for all }t \in (0, T_{\rm max})
    \end{align}
      This, together with Lemma \ref{L.2} implies that 
    \begin{align}
        \left \| v (\cdot,t )\right \|_{L^\infty(\Omega)} \leq C, \qquad \text{for all }t \in (0, T_{\rm max})
    \end{align}
    for some $C=C(M)>0$. This implies that $\inf_{(x,t)}D(v(x,t)  >0$ and therefore the degeneracy of the diffusion term is now eliminated. It follows that
\begin{align}\label{l2.2}
    - \frac{2(p-1)}{p} \int_\Omega D(v)|\nabla u^{\frac{p}{2}}|^2  \leq -c_1\int_\Omega |\nabla u^{\frac{p}{2}}|^2,
\end{align}
where $c_1=\frac{2p-2}{p}\inf_{(x,t)\in \Omega \times (0,T)} D(v(x,t))$.
    By integration by parts and the condition $S' \geq 0$, we have
    \begin{align} \label{l2.3}
        (p-1) \int_\Omega S(v) u^{p-1} \ln^\alpha (u+e) \nabla u \cdot \nabla v &= -c_2\int_\Omega S(v) \phi(u) \Delta v - c_2\int_\Omega S'(v) \phi(u)|\nabla v|^2 \notag \\
        &\leq c_3 \int_\Omega u \phi(u) \notag \\
        &\leq c_3\int_\Omega u^{p+1} \ln^\alpha (u+e).
    \end{align}
    where 
    \begin{align}
        \phi(l):= \int_0 ^l s^{p-1}\ln^\alpha(s+e)\, ds \leq l^p \ln^\alpha(l+e), \qquad \text{for all } l\geq 0.
    \end{align}
    From Lemma \ref{l1}, we obtain that $\sup_{t\in (0, T_{\rm max})} \int_\Omega u \ln(u+e)< \infty$. Now applying Lemma \ref{ILGN} with 
   \[\epsilon = \frac{c_1}{2   c_3\sup_{t\in (0, T_{\rm max})} \int_\Omega u \ln(u+e)} \]
    yields
    \begin{align} \label{l2.4}
        c_3\int_\Omega u^{p+1} \ln^\alpha (u+e) 
        &\leq   c_3\epsilon \int_\Omega |\nabla u^{\frac{p}{2}}|^2  \cdot \int_\Omega u \ln (u+e) +\epsilon   c_3\int_\Omega u  \cdot \int_\Omega u \ln(u+e)+c_4 \notag \\
        &\leq    \frac{c_1}{2} \int_\Omega |\nabla u^{\frac{p}{2}}|^2 +c_5
    \end{align}
    where $c_4 =C(\epsilon)>0$ and $c_5 =C(\epsilon)>0$.
    By Young's inequality, we obtain that
    \begin{align} \label{l2.5}
       \left ( r + \frac{1}{p} \right ) \int_\Omega u^p \leq \frac{\mu}{2} \int_\Omega u^{p+1} +c_6.
    \end{align}
    where $c_6=C(r,p,\mu)>0$.
    Collecting \eqref{l2.1}, \eqref{l2.2}, \eqref{l2.3} and \eqref{l2.5} yields
    \begin{align}
        \frac{1}{p} \frac{d}{dt} \int_\Omega u^p + \frac{1}{p} \int_\Omega u^p\leq c_7,
    \end{align}
    where $c_7 =C(\epsilon,p,\mu , r)>0$. Finally, we apply Gronwall's inequality to complete the proof.
\end{proof}

We are now ready to prove the main theorems
\begin{proof}[Proof of Theorem \ref{degenerate}]
      By using Lemma \ref{l2} for a fixed $p> 2$, it follows that 
    \[
    \sup_{t \in (0, T_{\rm max})} \left \|u(\cdot,t) \right \|_{L^p(\Omega)} <\infty.
    \]
    By elliptic regularity theory in Sobolev spaces, we obtain that 
     \begin{align}
          \sup_{t \in (0, T_{\rm max})} \left \| v(\cdot,t) \right \|_{W^{1,\infty}(\Omega)} <\infty.
     \end{align}
   Applying  Moser-Alikakos iteration (see e.g \cite{Winkler-2011, Alikakos1, Alikakos2} ) yields
 \[
    \sup_{t \in (0, T_{\rm max})} \left \| u(\cdot,t) \right \|_{L^\infty(\Omega)} <\infty.
    \]
    This, together with Lemma \ref{local-existence} implies that $T_{\rm max} = \infty$, which finishes the proof.
\end{proof}
\section{Fully Parabolic system} \label{P}
We will follow the framework established in the previous section to prove Theorem \ref{degenerate-p}. However, for the fully parabolic system, we cannot directly use the equation $\Delta v = u-v$ for estimations, as done in Lemmas \ref{l1} and \ref{l2}. Instead, we need to establish an intermediate estimate to connect the two equations of \eqref{1}. Let us commence this section with the following lemma.
\begin{lemma} \label{l}
    For any $p >1$, there exist positive constants $A_1,A_2,A_3$ depending only on $p$ such that 
    \begin{align} \label{l-1}
        \frac{1}{2p}\frac{d}{dt}\int_\Omega |\nabla v|^{2p} +A_1\int_\Omega \left | \nabla |\nabla v|^p \right |^2+ \int_\Omega |\nabla v|^{2p} \leq A_2 \int_\Omega u^2|\nabla v|^{2p-2} +A_3 \int_\Omega |\nabla v|^{2p}
    \end{align}
\end{lemma}
\begin{proof}
    We make use of the following point-wise identity
\begin{align*}
    \nabla v \cdot \nabla \Delta v = \frac{1}{2} \Delta( |\nabla v|^2) - |D^2v|^2 
\end{align*}
to obtain
\begin{align} \label{L2ets6}
    \frac{1}{2p}\frac{d}{dt}\int_\Omega |\nabla v|^{2p}+ \int_\Omega |\nabla v|^{2p}&= -c_1\int_\Omega |\nabla|\nabla v|^p|^2-\int_\Omega |\nabla v|^{2p-2} |D^2 v|^2 \notag\\
    &+\int_\Omega |\nabla v|^{2p-2}\nabla v\cdot \nabla u \notag \\
    & +c_2\int_{\partial \Omega} \frac{\partial |\nabla v|^2 }{\partial \nu}|\nabla v|^{2p-2},
\end{align}
where $c_1, c_2$ are positive constants depending only on $p$.  The inequality $\frac{\partial |\nabla v|^{2}}{\partial \nu} \leq M |\nabla v|^2 ,$ (see \cite{Souplet-13}[Lemma 4.2])  for some $M>0$ depending only on $\Omega$, implies that
\[
c_2\int_{\partial \Omega} \frac{\partial |\nabla v|^2 }{\partial \nu}|\nabla v|^{2p-2}\, dS \leq c_2M\int_{\partial \Omega }|\nabla v|^{2p} \, dS  .
\]
Let $g:=|\nabla v|^q$ and apply Trace Imbedding Theorem $W^{1,1}(\Omega) \longrightarrow L^1(\partial \Omega)$ together with Young's inequality, there exists positive constants $C$ and $c_3$ such that
\begin{align}
    c_2M\int_{\partial \Omega} g^2 \, dS &\leq C \int_\Omega g|\nabla g| +C \int_\Omega g^2 \notag \\
    &\leq \frac{c_1}{2} \int_\Omega |\nabla g|^2 +c_3\int_\Omega g^2,
\end{align}
Therefore, we have
\begin{align} \label{nonconvex-1}
    c_2M\int_{\partial \Omega }|\nabla v|^{2p} \, dS \leq \frac{c_1}{2} \int_{\Omega}|\nabla|\nabla v|^p|^2 +c_3 \int_\Omega |\nabla v|^{2p}.  
\end{align} Applying the pointwise inequality $(\Delta v)^2 \leq 2 |D^2 v|^2$ to \eqref{L2ets6}  yields
\begin{align} \label{L2ets7}
    \frac{1}{2p}\frac{d}{dt}\int_\Omega |\nabla v|^{2p}+ \int_\Omega |\nabla v|^{2p}  &\leq   - \frac{c_1}{2}\int_\Omega |\nabla|\nabla v|^{p}|^2 -\frac{1}{2}\int_\Omega |\nabla v|^{2p-2} |\Delta v|^2 \notag\\
    &+\int_\Omega |\nabla v|^{2p-2}\nabla v\cdot \nabla u +c_{3}\int_\Omega |\nabla v|^{2p} 
\end{align}
 By integration by parts and elemental inequalities, there exist constants $c_{4} =C(p)>0$ and $c_5=C(p)>0$ in such a way that
\begin{align}\label{L2ets8}
    \int_\Omega|\nabla v|^{2p-2}\nabla v \cdot \nabla u &= -\int_\Omega u|\nabla v|^{2p-2}\Delta v  -c_4\int_\Omega u |\nabla v|^{p-1} \nabla |\nabla v|^p \cdot \frac{ \nabla v}{| \nabla v|} \notag \\ 
    &\leq \frac{1}{2} \int_\Omega (\Delta v)^2 |\nabla v|^{2p-2} + \frac{c_1}{4}\int_\Omega |\nabla |\nabla v|^p|^2  \notag \\
    &+ c_{5}\int_\Omega u^2 |\nabla v|^{2p-2},
\end{align}
From \eqref{L2ets7} and \eqref{L2ets8}, we finally prove \eqref{l-1}.

\end{proof}
The following lemma, akin to Lemma \ref{l1}, provides a crucial a priori estimate for solutions. However, the constant $k$ is now bounded from above due to the structure of parabolic equations. In addition to the estimate for $\int_\Omega u \ln^k(u+e)$, we also require a uniform bound in time for $\int_t^{t+\tau} \int_\Omega u^2 \ln^k(u+e)$ to cooperate with Proposition \ref{Orlicz-regularity} in order to obtain uniform bounds for $v$.
\begin{lemma} \label{l1-p}
    Under the assumptions in Theorem \ref{degenerate-p}, for any $k\in (1,2-2\alpha)$, we have that 
      \begin{equation}
       \sup_{t\in (0,T_{\rm max})}\int_\Omega \left \{  u \ln^k{(u+e)} + |\nabla v|^2  \right \} + \sup_{t \in (0, T_{\rm max}-\tau)} \int_t ^{t+\tau} \int_\Omega  u^2 \ln^{k}(u+e) <\infty,
   \end{equation}
   where $\tau = \min \left \{ 1, \frac{T_{\rm max}}{2}\right \}$.
\end{lemma}
\begin{proof}
        We define
 \[
 y(t):= \int_\Omega u \ln^k{(u+e)} + \frac{1}{2} \int_\Omega |\nabla v|^2,
 \]
 and differentiate $y(\cdot)$ to obtain
 \begin{align} \label{l1-p.1}
    y'(t)&= \int_\Omega \left ( \ln^k{(u+e)} +ku \frac{\ln^{k-1}{(u+e)}}{u+e} \right ) u_t + \int_\Omega \nabla v \cdot \nabla v_t  \notag \\
    & := I'(t)+J'(t).
 \end{align}
Where $I$ is given in Lemma \ref{l1}. We now just reuse estimations from \eqref{l1-p.2} to \eqref{l1-p.6} for $I'$ except for \eqref{l1.1}. By using integration by parts, taking into account the condition $S' \geq 0$ and applying elementary inequalities, we obtain that 
\begin{align}\label{l1-p.4}
   \int_\Omega S(v) \nabla \phi(u)\cdot \nabla v&=-\int_\Omega S(v) \phi(u) \Delta v - \int_\Omega S'(v) \phi(u)|\nabla v|^2 \notag \\
    &\leq  \left \| S \right \|_{L^\infty(0,\infty)}\int_\Omega \phi(u)|\Delta v| \notag \\
    &\leq \frac{1}{2}  \int_\Omega (\Delta v)^2 +\frac{\left \| S \right \|_{L^\infty(0,\infty)}^2}{2} \int_\Omega \phi^2(u) \notag \\
    & \leq \frac{1}{2} \int_\Omega (\Delta v)^2 +\frac{c_1\left \| S \right \|_{L^\infty(0,\infty)}^2}{2}\int_\Omega u^2 \ln^{2k+2\alpha-2}(u+e) \notag \\
    &\leq \frac{1}{2} \int_\Omega (\Delta v)^2 +\frac{\mu}{4} \int_\Omega u^2 \ln^{k}(u+e) +c_2,
\end{align}
 where $\phi$ is defined in \eqref{l1-p.3},  $c_2 =C(\mu)>0$ and the last inequality comes from the fact that $2k+2\alpha-2<k$ and the inequality \eqref{l1-p.5}.  Collecting \eqref{l1-p.2}, \eqref{l1-p.4} and \eqref{l1-p.6}, we have
\begin{align} \label{l1-p.6'}
    I'(t) \leq \frac{1}{2} \int_\Omega (\Delta v)^2- \frac{\mu}{2}\int_\Omega u^2 \ln^k(u+e) +c_4,
\end{align}
where $c_4 =C(\mu)>0$.  By integration by parts and elemental inequalities, it follows that there exist $c_5=C(\mu)>0$ such that
\begin{align} \label{l1-p.7}
    J'(t) &:= \int_\Omega \nabla v \cdot \nabla v_t \notag \\
    &=-\int_\Omega (\Delta v)^2 -\int_\Omega|\nabla v|^2 -\int_\Omega u\Delta v \notag \\
    &\leq -\frac{1}{2}\int_\Omega (\Delta v)^2 -\int_\Omega|\nabla v|^2  +\frac{1}{2 }\int_\Omega u^2 \notag \\
    &\leq -\frac{1}{2}\int_\Omega (\Delta v)^2 -\int_\Omega|\nabla v|^2  +\frac{\mu}{4} \int_\Omega u^2 \ln^{k}(u+e)+c_5.
\end{align}
The inequality \eqref{l1-p.5} implies that there exists $c_{6}=C(\mu)>0$ such that
\begin{align} \label{l1-p.8}
    \int_\Omega u \ln^{k}(u+e) \leq \frac{\mu}{8}  \int_\Omega u^2 \ln^{k}(u+e)+c_{6}.
\end{align}
Collecting \eqref{l1-p.1}, \eqref{l1-p.6'}, \eqref{l1-p.7}, and \eqref{l1-p.8} we obtain 
\begin{align}
    y'(t)+y(t) +\frac{\mu}{8}\int_\Omega u^2 \ln^k(u+e) \leq c_7
\end{align} 
for some $c_7=C \left ( r,k, \mu, \left \| S\right \|_{L^\infty( (0, \infty))} \right )>0$. Applying Gronwall's inequality to this leads to $y(t)\leq \max \left \{ y(0), c_{7} \right \}$.
 Additionally, we also have:
\begin{align}\label{general.ie.14}
  \frac{\mu}{8} \int_\Omega u^2 \ln^{k}(u+e) \leq c_{11} -y'(t).
\end{align}
By integrating the previous inequality from $t$ to $t+\tau$ and using the fact that $y$ is non-negative and bounded, we can conclude the proof.
\end{proof}

Now, we can establish $L^p$ bounds for solutions in the following lemma, akin to Lemma \ref{l2}.

\begin{lemma} \label{l2-p}
    Under the assumption in Theorem \ref{degenerate-p}, for any $p>\max \left \{ \frac{\alpha}{1-2\alpha}, 1 \right \}$, we have that 
    \begin{align} \label{l2-p-1}
        \sup_{t\in (0, T_{\rm max })} \int_\Omega \left \{  u^p(\cdot,t) + |\nabla v (\cdot,t)|^{2p}  \right \} dx<\infty.
    \end{align}
\end{lemma}
\begin{proof}
       We define 
\begin{equation*}
    \phi(t):= \frac{1}{p} \int_\Omega u^p +\frac{1}{2p} \int_\Omega |\nabla v|^{2p},
\end{equation*}
and differentiate $\phi(\cdot)$ to obtain:
\begin{align} \label{l2-p.1}
    \phi'(t)&= \int_\Omega u^{p-1}\left [ \nabla \cdot (D(v)\nabla u) - \nabla \cdot (S(v)u \ln^\alpha (u+e) \nabla v)  +ru-\mu u^2  \right ] \notag \\
    &+\int_\Omega |\nabla v|^{2p-2} \nabla v \cdot \nabla  \left ( \Delta v +u- v \right ) \notag \\
    &:= M_1+M_2.
\end{align}
By integration by parts, we have
  \begin{align} \label{l2-p.2}
        M_1 &= - \frac{2(p-1)}{p} \int_\Omega D(v)|\nabla u^{\frac{p}{2}}|^2 +(p-1) \int_\Omega S(v) u^{p-1} \ln^\alpha (u+e) \nabla u \cdot \nabla v \notag \\
        &+r\int_\Omega u^p -\mu \int_\Omega u^{p+1}. 
    \end{align}
  From Lemma \ref{l1-p}, we find that
    \begin{align*}
        \sup_{t\in (0, T_{\rm max} -\tau )} \int_\Omega u^2 \ln^k(u+e) < \infty
    \end{align*}
    for any $k \in (1,2-2\alpha)$. This, together with Proposition \ref{Orlicz-regularity} implies that
    \begin{align}
        \sup_{t\in (0, T_{\rm max})} \left \| v(\cdot,t) \right \|_{L^\infty(\Omega)} < \infty.
    \end{align}
    Therefore, it follows that $\inf_{(x,t)\in \Omega \times (0,T)} D(v(x,t)) >0$ and
\begin{align}\label{l2-p.3}
    - \frac{2(p-1)}{p} \int_\Omega D(v)|\nabla u^{\frac{p}{2}}|^2  \leq -c_1\int_\Omega |\nabla u^{\frac{p}{2}}|^2,
\end{align}
where $c_1=\frac{2p-2}{p}\inf_{(x,t)\in \Omega \times (0,T)} D(v(x,t))$. Since $p> \max \left \{1, \frac{\alpha}{1-2\alpha} \right \}$ we can fix \[ k \in \left (  \max \left \{ \frac{2\alpha (p+1)}{p} , 1\right \}, 2- 2\alpha \right )\]
and Lemma \ref{l1-p} allows us to choose
\begin{align} \label{epsilon}
    \epsilon = \min \left \{ \frac{A_1}{2C_{GN}\sup_{t \in (0, T_{\rm max})} \int_\Omega |\nabla v|^2}; \frac{c_1}{2\sup_{t \in (0, T_{\rm max})}\int_\Omega u \ln^k(u+e)}  \right \},
\end{align}
where $A_1$ is the constant defined in Lemma \ref{l}. By Young's inequality, we obtain
    \begin{align}\label{l2-p.4}
       (p-1) \int_\Omega S(v) u^{p-1} \ln^\alpha (u) \nabla u \cdot \nabla v &= \frac{2(p-1)}{p} \int_\Omega S(v) u^{\frac{p}{2}} \ln^\alpha (u+e) \nabla u^{\frac{p}{2}} \cdot \nabla v \notag \\
            &\leq \frac{c_1}{4} \int_\Omega |\nabla u^{\frac{p}{2}}|^{2} + \frac{4(p-1)^2}{p^2c_1} \int_\Omega u^{p} \ln^{2\alpha }(u+e)|\nabla v|^2  \notag \\
            &\leq \frac{c_1}{4} \int_\Omega |\nabla u^{\frac{p}{2}}|^{2} + \frac{\epsilon}{2} \int_\Omega |\nabla v|^{2p+2}+ c_2\int_\Omega u^{p+1}\ln^{\frac{2\alpha(p+1)}{p}}(u+e),
    \end{align}
    where $c_2= \frac{8(p-1)^2}{p^2c_1\epsilon}$. 
    Combining \eqref{l2-p.2}, \eqref{l2-p.3}, and \eqref{l2-p.4} yields
    \begin{align}  \label{l2-p.5'}
        M_1 \leq -\frac{3c_1}{4}\int_\Omega |\nabla u^{\frac{p}{2}}|^2 +\frac{\epsilon}{2}\int_\Omega |\nabla v|^{2p+2}+c_2\int_\Omega u^{p+1}\ln^{\frac{2\alpha(p+1)}{p}}(u+e) +r\int_\Omega u^p.
    \end{align}
    From Lemma \ref{l}, we have
    \begin{align} \label{l2-p.5}
        M_2 +A_1\int_\Omega \left | \nabla |\nabla v|^p \right |^2+ \int_\Omega |\nabla v|^{2p} \leq A_2 \int_\Omega u^2|\nabla v|^{2p-2} +A_3 \int_\Omega |\nabla v|^{2p}.
    \end{align}
    By elementary inequalities, we obtain that
    \begin{align}\label{l2-p.6}
     \left (r +\frac{1}{p} \right )\int_\Omega u^p  +  A_2 \int_\Omega u^2|\nabla v|^{2p-2} +A_3 \int_\Omega |\nabla v|^{2p} \leq \frac{\epsilon}{2} \int_\Omega |\nabla v|^{2p+2} + c_3 \int_\Omega u^{p+1}\ln^{\frac{2\alpha(p+1)}{p}}(u+e) +c_4,
    \end{align}
    where $c_3=C(\epsilon)>0$ and $c_4=C(\epsilon)>0$. Collecting \eqref{l2-p.1}, \eqref{l2-p.5'}, \eqref{l2-p.5}, and \eqref{l2-p.6} yields
    \begin{align}\label{l2-p.7}
        \phi'(t)+\phi(t)  &\leq -\frac{3c_1}{4} \int_\Omega |\nabla u^{\frac{p}{2}}|^2 -A_1\int_\Omega \left | \nabla |\nabla v|^p \right |^2  +\epsilon \int_\Omega |\nabla v|^{2q+2} \notag \\
        &+ c_5\int_\Omega u^{p+1}\ln^{\frac{2\alpha(p+1)}{p}}(u+e) +c_4,
    \end{align}
    where $c_5=c_2+c_3$. Using the Gagliardo-Nirenberg interpolation inequality for $n=2$ and the fact that $\sup_{t\in (0, T_{\rm max})}\int_\Omega |\nabla v(\cdot,t)|^2 \, dx <\infty$ from Lemma \ref{l1-p}, there exists a positive constant $C_{GN}$ such that:
\begin{align} \label{l2-p.8}
     \epsilon \int_\Omega |\nabla v|^{2p+2} &\leq  \epsilon C_{GN}\int_\Omega |\nabla |\nabla v|^p|^2  \int_\Omega  |\nabla v|^2 +   \epsilon C_{GN}\left (\int_\Omega  |\nabla v|^2 \right )^{p+1} \notag \\
    &\leq c_{6} \epsilon \int_\Omega |\nabla |\nabla v|^p|^2  + c_{7},
\end{align}
     where $c_6 = C_{GN}\sup_{t\in (0,T_{\rm max})} \int_\Omega |\nabla v(\cdot,t)|^2$ and $c_{7} = \epsilon C_{GN}\sup_{t\in (0,T_{\rm max})} \left ( \int_\Omega |\nabla v(\cdot,t)|^2 \right )^{p+1}$. The condition $\frac{2\alpha(p+1)}{p} <k< 2-2\alpha$ when $p> \max \left \{ \frac{\alpha}{1-2\alpha}, 1 \right \}$ enables us to apply Lemma \ref{ILGN} to obtain
    \begin{align}\label{l2-p.9}
        c_5\int_\Omega u^{p+1}\ln^{\frac{2\alpha(p+1)}{p}}(u) &\leq \epsilon \int_\Omega |\nabla u^{\frac{p}{2}}|^2 \int_\Omega u \ln^k(u) +\epsilon  \left ( \int_\Omega u \right )^p\int_\Omega u\ln^k(u+e) +c_7 \notag \\
        &\leq c_8 \epsilon \int_\Omega |\nabla u^{\frac{p}{2}}|^2 +c_9
    \end{align}
    where $c_8=\sup_{t\in (0, T_{\rm max }} \int_\Omega u \ln ^k (u+e))$, and $c_9=c(\epsilon)>0$. From \eqref{l2-p.7}, \eqref{l2-p.8}, and \eqref{l2-p.9}, we have
    \begin{align*}
         \phi'(t)+\phi(t)  &\leq \left ( c_8\epsilon- \frac{3c_1}{4} \right )\int_\Omega |\nabla u^{\frac{p}{2}}|^2 + \left (c_6 \epsilon - A_1 \right )\int_\Omega |\nabla |\nabla v|^p|^2  \int_\Omega  |\nabla v|^2 +c_{10},
    \end{align*}
    where $c_{10}= c_4+c_9$. From \eqref{epsilon}, we find that $c_8\epsilon -\frac{3c_1}{4} \leq 0$, and $c_6\epsilon -A_1 \leq 0$. It follows that  $\phi'(t)+\phi(t) \leq c_{10}$. The proof is finished by applying Gronwall's inequality.
\end{proof}
\begin{proof} [Proof of Theorem \ref{degenerate-p}]
      By using Lemma \ref{l2-p} for a fixed $p> 2$, it follows that 
    \[
    \sup_{t \in (0, T_{\rm max})} \left \|u(\cdot,t) \right \|_{L^p(\Omega)} <\infty.
    \]
    By Lemma \ref{Para-Reg}, we have that
     \begin{align} \label{thm1}
          \sup_{t \in (0, T_{\rm max})} \left \| v(\cdot,t) \right \|_{W^{1,\infty}(\Omega)} <\infty.
     \end{align}
    Now by applying Moser-Alikakos iteration procedure, we obtain 
\begin{align} \label{thm2}
     \sup_{t \in (0, T_{\rm max})} \left \| u(\cdot,t) \right \|_{L^\infty(\Omega)} <\infty.
\end{align}
    This, together with \eqref{thm1} and Lemma \ref{local-existence} implies that $T_{\rm max} = \infty$, which completes the proof.
    
\end{proof}
\section{Acknowledgments}
The author expresses gratitude to Professor Michael Winkler for sharing his preprint manuscript, which includes a crucial tool for the findings presented in this paper. Additionally, the author acknowledges support from the Mathematics Graduate Research Award Fellowship at Michigan State University.

\printbibliography
\end{document}